\documentclass[english]{article}
\usepackage[T1]{fontenc}
\usepackage[latin9]{luainputenc}
\usepackage{amsmath}
\usepackage{amsthm}
\usepackage{amssymb}

\makeatletter
\theoremstyle{plain}
\newtheorem{thm}{\protect\theoremname}
  \theoremstyle{plain}
  \newtheorem{cor}[thm]{\protect\corollaryname}

\makeatother

\usepackage{babel}
  \providecommand{\corollaryname}{Corollary}
\providecommand{\theoremname}{Theorem}

\begin{document}

\title{Integral and Absolute Hodge Classes}

\author{Ryan Keast}
\maketitle
\begin{abstract}
Despite the failure of the integral Hodge conjecture, we show that
the rational Hodge conjecture implies an integral version (modulo
torsion) of the absolute Hodge conjecture.
\end{abstract}

\section{Introduction}

This very brief article is based on a talk given by the author and
the proceeding discussion after the lecture. The author would like
to thank Donu Arapura for a very important insight.

Andre-Oort predicts the algebraic closure of CM points inside a Shimura
variety is again Shimura variety. This was proven in the special case
of $A_{g}$ by J. Tsimerman using the combined tools of o-minimality
and large Galois orbits of CM abelian varieties\cite{T}. Effective
CM Hodge structures of positive weight are always embeddable into
the Hodge structure of CM abelian varieties\cite{A}. A natural question
is if we can use this embedding to produce large Galois orbits for
a whole host of varieties. The immediate road block is that it is
unknown whether or not this embedding is algebraic, or even absolute.

Even when we assume the Hodge conjecture there seemed to be a gap.
Let $X$ and $Y$ be smooth projective varieties with $H^{n}(X,\mathbb{Z})$,
$H^{n}(Y,\mathbb{Z})$ both free. Assume there exists an integral
isomorphism $j$ of Hodge structures. The absolute Hodge conjecture
implies that after applying $\sigma\in Auto(\mathbb{C})$, $\sigma j$
is a rational morphism of Hodge structures. What is not clear apriori
is if $\sigma j$ is still an integral isomorphism. If we are to establish
a relationship between the sizes of the Galois orbits, integrality
must be preserved, else we could also be counting isogenies.

Surprisingly the literature gave no answer to the question. It appeared
to be silent on an even more basic question: if $v\in H^{n,n}(X,\mathbb{C})\cap H^{2n}(X,\mathbb{Z})/torsion$,
is $\sigma v\in$$H^{n,n}(X^{\sigma},\mathbb{C})\cap H^{2n}(X^{\sigma},\mathbb{Z})/torsion$.
If we assumed the integral Hodge conjecture, then $v=\sum n[Z]$ $n\in\mathbb{Z}$
and $\sigma v=\sum n[Z^{\sigma}]$. Alas, the integral Hodge conjecture
is false.

Counterexamples to the integral Hodge conjecture lie in two camps:
torsion and non-torsion. We will ignore the torsion counterexamples
for now. A basic non-torsion counterexample to the integral Hodge
conjecture was given by Kollar(cf\cite{V}). If we have of a very
general threefold hypersurface of degree 125, then 5 divides the degree
of every curve. $H^{2}(X,\mathbb{Z})\cap H^{1,1}(X,\mathbb{C})$ is
1-dimension and given by cohomology class of a hyperplane section
$h$. Assume $H^{4}(X,\mathbb{Z})\cap H^{2,2}(X,\mathbb{C})$ is generated
by curve $c$. 

Since the pairing of $H^{2}$and $H^{4}$ is unimodular$\int h\cap[c]=1$.
This integral also gives the intersection number, so it must be divisible
by five. We arrive at a contradiction. 

In this case, the rational Hodge conjecture is known for $H^{4}.$
Let $c$ be a curve such that $[c]$ generates $H^{4}(X,\mathbb{Q})\cap H^{2,2}(X,\mathbb{C})$.

$\int h\cap[c]=n$. Then $\frac{1}{n}[c]$ is integral generates $H^{4}(X,\mathbb{Z})\cap H^{2,2}(X,\mathbb{C})$.
$\int\sigma h\cap[\sigma c]=\int h\cap[c]=n$, so $\frac{1}{n}[c^{\sigma}]$
also generates $H^{4}(X^{\sigma},\mathbb{Z})\cap H^{2,2}(X^{\sigma},\mathbb{C})$.
Thus even in this counterexample of the integral Hodge conjecture,
the integrality is preserved by $\sigma$.

This example convinced the author that the preservation of integrality
was at the very least not trivially false. In fact, in the proceeding
section we will show that despite the failure of the integral Hodge
conjecture the integrality of $\sigma j$ is guaranteed by the rational
Hodge conjecture. In other words, the rational conjecture implies
an integral version of the absolute Hodge conjecture modulo torsion. 

\section{Absolute and Integral Hodge Classes}

The proof involves passing to the étale cohomology: 
\begin{thm}
Assuming the rational Hodge conjecture, if $v\in H^{n,n}(X,\mathbb{C})\cap H^{2n}(X,\mathbb{Z})/torsion$,
then for $\sigma\in Aut(\mathbb{C})$ $\sigma v\in$$H^{n,n}(X^{\sigma},\mathbb{C})\cap H^{2n}(X^{\sigma},\mathbb{Z})/torsion$
\end{thm}

\begin{proof}
If we assume the rational Hodge conjecture, $v=\sum a_{i}[Z_{i}]$
with $a_{i}\in\mathbb{Q}$. $\sigma v=\sum a_{i}[Z_{i}^{\sigma}]$.
By the Artin Comparison theorem\cite{M}, for each $l$ we have a
canonical isomorphism $A_{l}:$$H^{2n}(X(\mathbb{C}),\mathbb{Z}_{l})\rightarrow H^{2n}(X_{et,}\mathbb{Z}_{l})$.
We have the embedding$H^{2n}(X(\mathbb{C}),\mathbb{Z})\hookrightarrow H^{2n}(X(\mathbb{C}),\mathbb{Z}_{l})$.
The image of $v=\sum a_{i}[Z_{i}]$ and $\sigma v=\sum a_{i}[Z_{i}^{\sigma}]$
are identical with respect to the étale cohomology, specifically they
are both integral with respect to $\mathbb{Z}_{l}$ in $H^{2n}(X_{et,}\mathbb{Z}_{l})\cong H^{2n}(X_{et,}^{\sigma}\mathbb{Z}_{l}).$
It follows that $\sigma v=\sum a_{i}[Z_{i}^{\sigma}]\in H^{2n}(X^{\sigma}(\mathbb{C}),\mathbb{Z}_{l})$.
Let $\{e_{n}\}$ be a basis for $H^{n,n}(X^{\sigma},\mathbb{C})\cap H^{2n}(X^{\sigma},\mathbb{Z})/torsion$.
We rewrite $\sigma v=\sum b_{n}e_{n}$. We will now show that $b_{n}$
are integers. If $\sum b_{n}e_{n}\in$$H^{2n}(X^{\sigma}(\mathbb{C}),\mathbb{Z})\otimes Z_{l}$,
then each $b_{n}$ must be an $l$-adic integer. Since this is true
for all $l$, $b_{n}\in\mathbb{Z}$.
\end{proof}
\begin{cor}
Assume the rational Hodge conjecture. If $H^{n}(X,\mathbb{C})$ and
\textup{$H^{n}(Y,\mathbb{C})$} are free and isomorphic as integral
Hodge structures, then so are $H^{n}(X^{\sigma},\mathbb{C})$ and
\textup{$H^{n}(Y^{\sigma},\mathbb{C}).$}
\end{cor}

\begin{proof}
By assumption, there exists a integral morphism of Hodge structures
$j:H^{n}(X,\mathbb{C})\rightarrow H(Y,\mathbb{C})$ with an integral
inverse $j^{-1}$. If we assume the rational Hodge conjecture $j$
is given by an algebraic correspondence, $\sigma j$ and $\sigma j^{-1}$
are still inverses of each other and are still both integral.
\end{proof}

\end{document}